\documentclass[12pt]{amsart}

\usepackage{amssymb,amsmath,amscd,amsthm,enumerate,verbatim,color}
\usepackage[utf8]{inputenc}
\usepackage[mathscr]{euscript} 
\usepackage{bbm}
\usepackage[left=1.3in,right=1.3in,top=1in,bottom=1in]{geometry}

\definecolor{purple}{rgb}{.5,0,1}
\definecolor{orange}{rgb}{1,.5,0}
\definecolor{pink}{rgb}{1,0,.5}
\definecolor{green}{rgb}{0,.5,0}
\definecolor{gold}{rgb}{1,.623,0}

\date{}

\newcommand{\A}{{\mathcal A}}

\newcommand{\Z}{{\mathbbm Z}}
\newcommand{\R}{{\mathbbm R}}

\newcommand{\N}{{\mathbbm N}}

\renewcommand{\S}{{\mathbbm S}}

\newcommand{\e}{{\varepsilon}}

\newcommand{\Dir}{{\mathrm{D}}}
\newcommand{\Neu}{{\mathrm{N}}}

\def\a{{\sf a}}
\def\b{{\sf b}}
\def\ab{{\a\b}}

\newcommand{\dd}{{\mathrm{d}}}

\newcommand{\loc}{{\mathrm{loc}}}

\newcommand{\Hd}{{\mathrm{H}}}

\newtheorem{theorem}{Theorem}[section]

\newtheorem{prop}[theorem]{Proposition}

\theoremstyle{definition}
\newtheorem{remark}[theorem]{Remark}
\newtheorem*{remark*}{Remark}
\theoremstyle{definition}

\theoremstyle{definition}

\theoremstyle{definition}

\newtheorem*{questt}{Question 6.8}

\allowdisplaybreaks

\numberwithin{equation}{section}

\newcommand{\tr}{\mathrm{tr} }

\newcommand{\set}[1]{\left\{#1\right\}}
\newcommand{\eqdef}{\overset{\mathrm{def}}=}

\begin{document}

\title[Spectral Properties of Continuum Fibonacci Operators]{Spectral Properties of Continuum Fibonacci Schr\"odinger Operators}

\author[J.\ Fillman]{Jake Fillman}
\email{fillman@vt.edu}
\thanks{J.\ F.\ was supported in part by an AMS-Simons travel grant, 2016--2018}

\author[M.\ Mei]{May Mei}
\email{meim@denison.edu}
\thanks{M.\ M.\ was supported in part by Denison University and a Woodrow Wilson Career Enhancement Fellowship}

\begin{abstract}
We study continuum Schr\"odinger operators on the real line whose potentials are comprised of two compactly supported square-integrable functions concatenated according to an element of the Fibonacci substitution subshift over two letters. We show that the Hausdorff dimension of the spectrum tends to one in the small-coupling and high-energy regimes, regardless of the shape of the potential pieces.
\end{abstract}

\maketitle



\section{Introduction}\label{s:intro}

\subsection{Background}

Quasicrystals were discovered in the early 1980s by D.\ Schechtman \textit{et al.}\ \cite{SBGC84} and have attracted a substantial amount of attention from researchers in mathematics and science. Broadly speaking, quasicrystals are solids characterized by the coexistence of two characteristics: \emph{aperiodicity} (i.e.\ the absence of translation symmetries) and \emph{long-range order}.
 
Thus, objects generated by or associated with strictly ergodic aperiodic subshifts over finite alphabets furnish concrete mathematical models of quasicrystals. Researchers in mathematics have studied such operators from various points of view including diffraction theory (see \cite{BBG2015} and references therein) and spectral theory (see \cite{Damanik2007:subshifts,DEG2015} and references\ therein). In particular, self-adjoint operators generated by such subshifts (and more recently, unitary operators) have been studied fairly heavily since the 1980s. Until somewhat recently, most of the effort was devoted to discrete Schr\"odinger operators arising in this manner. On the other hand, there have been several recent investigations concerning continuum Schr\"odinger operators \cite{DFG2014,KLS2011,LSS2014}, Jacobi matrices \cite{FTY2016,Yessen2013}, and CMV matrices \cite{DFO2016,DFV2014,DamLenz2007, DMY2013JAT,DMY2013JSP,F2017PAMS,MO2014}.

The (discrete) Fibonacci Hamiltonian is the central model of a 1-dimensional quasicrystal in the discrete Schr\"odinger setting. This model was proposed in the physics literature by Kohmoto--Kadanoff--Tang \cite{KKT1983} and Ostlund \textit{et al.}\ \cite{OPRSS1983}. Casdagli and S\"ut\H{o} \cite{Cas1986CMP,Suto87,Suto89} wrote the seminal mathematics papers on this model, proving that it enjoys purely singular continuous spectrum supported on a Cantor set of zero Lebesgue measure. This spectral type is regarded as characteristic for this class of models. Our understanding of the discrete Fibonacci Hamiltonian is quite advanced; see \cite{DGY} and references therein. However, our knowledge about the continuum Fibonacci Hamiltonian is substantially more rudimentary, owing to a proliferation of nontrivial obstructions, such as unboundedness of the operator, non-constancy of the Fricke--Vogt character on the spectrum, inability to explicitly compute transfer matrices (and hence the Fricke--Vogt character) in all but the simplest cases, interalia.

Continuum versions of the Fibonacci Hamiltonian have been studied in \cite{BJK1992, DFG2014, Ghosh1991, Holzer1988, KN2002, KS1986, TK1989, WSS1988}. The general theory for such continuum operators was established in the papers \cite{DFG2014,LSS2014}. In the case of the Fibonacci subshift, \cite{DFG2014} also established asymptotic behavior of the local Hausdorff dimension of the spectrum in the regimes of large energy and small potentials whenever the potential pieces are given by characteristic functions of intervals of length one. Their work relied on explicit formulae and calculations in an essential fashion and hence could not be immediately generalized to other potentials. Consequently, they posed the following question:
\begin{questt}
Is it true that \lbrack these asymptotics\rbrack\ hold regardless of the shape of the bump? That is, if we replace $f_a=\lambda \cdot \chi_{[0,1)}$ and $f_b=0 \cdot  \chi_{[0,1)}$ by general $f_a \in L^2(0, \ell_a)$ and $f_b \in L^2(0, \ell_b)$, do \lbrack these asymptotics\rbrack\ continue to hold as stated?
\end{questt}
In this paper, we answer this question in the affirmative in the regimes of small coupling and large energy. As of right now, the regime of large coupling is out of reach. For the remainder of Section~\ref{s:intro}, we provide background on this class of operators and state our results at the end of the section.

\subsection{Schr\"odinger Operators Associated with Subshifts over Finite Alphabets} \label{s:subshift}
In this section, we will introduce the operators to be studied in the present work. First, we define the notion of concatenation of real-valued functions defined on intervals. Assume that for each $n \in \Z$, we have $\ell_n \in \R_+ \eqdef (0,\infty)$ and a real-valued function $f_n$ defined on $[0,\ell_n)$. Moreover, assume that
\begin{equation} \label{eq:lndiv}
\sum_{n\geq 0 } \ell_n 
= 
\sum_{n < 0} \ell_n 
= 
\infty.
\end{equation}
The condition \eqref{eq:lndiv} ensures that ensuing function will have domain $\R$. We define the \emph{concatenation} of the sequence $\{f_n\}_{n \in \Z}$ as follows. Put
\begin{equation} \label{eq:concatEndptDef}
s_n
\eqdef
\begin{cases}
\sum_{j=0}^{n-1} \ell_j & n \geq 1 \\
0 & n = 0 \\
-\sum_{j=n}^{-1} \ell_j & n \leq -1,
\end{cases}
\end{equation}
denote $I_n = [s_{n},s_{n+1})$, and define
\begin{equation}\label{eq:concatDef}
f(x)
=
f_n(x - s_{n}),\quad
\text{ for each } x \in I_n.
\end{equation}
We will denote this concatenation 
\begin{equation} \label{eq:finiteConcatDef}
f
=
\left( \cdots \; | \; f_{-2} \; | \; f_{-1} \; | \; \boxed{f_0} \; | \; f_1 \; | \; f_2 \; | \;  \cdots  \right).
\end{equation}
We use a box to indicate the position of the origin. One can also concatenate finite families of functions. Given $m\leq 0 \leq n$, and $\{f_j,\ell_j\}_{j=m}^n$ as above, we define $\left( f_{m}| \cdots |
\, \boxed{f_0} \, | \cdots |f_n \right)$ on $[s_m, s_{n+1})$ via \eqref{eq:concatEndptDef} and \eqref{eq:concatDef}.

Let $\mathcal{A}$ be a finite set, called the \emph{alphabet}. Equip $\mathcal{A}$ with the discrete topology and endow $\mathcal{A}^\Z$ with the corresponding product topology. The \emph{left shift} 
\[
[T \omega](n) \eqdef \omega(n+1),
\quad \omega \in \A^{\Z}, \; n \in \Z,
\]
defines a homeomorphism from $\A^{\Z}$ to itself. A subset $\Omega \subseteq \mathcal{A}^\Z$ is called \emph{$T$-invariant} if $T^{-1}(\Omega) = \Omega$. Any compact $T$-invariant subset of $\A^\Z$ is called a \emph{subshift}.

We can use the concatenation construction above to associate potentials (and hence Schr\"odinger operators) to elements of subshifts as follows. For each $\alpha \in \mathcal{A}$, we pick $\ell_\alpha > 0$ and a real-valued function $f_\alpha \in L^2[0,\ell_\alpha)$. Then, for any $\omega \in \A^\Z$, we define the continuum Schr\"odinger operator
\begin{equation}\label{eq:Schrod}
H_\omega 
= 
- \frac{\dd^2}{\dd x^2} + V_\omega
\end{equation}
in $L^2(\R)$ via 
\begin{equation} \label{eq:VomegaDef}
V_{\omega} 
= 
V_{\omega,f_\a,f_\b}
\eqdef
\left( \cdots \; f_{\omega_{-2}} \; | \; f_{\omega_{-1}} \; | \; \boxed{f_{\omega_0}} \; | \; f_{\omega_1} \; | \; f_{\omega_2} \; \cdots \right).
\end{equation}
These potentials belong to $L^2_\mathrm{loc, unif}(\R)$ and hence each $H_\omega$ defines a self-adjoint operator in $L^2(\R)$ in a canonical fashion.

\subsection{The Fibonacci Subshift}

In this paper, we study a special case of the foregoing construction, namely potentials generated by elements of the Fibonacci subshift. In this case, the alphabet contains two symbols, $\A \eqdef \{ \a,\b \}$. The Fibonacci substitution is the map
\[
S(\a) = \a\b, \; S(\b) = \a.
\]
This map extends by concatenation to $\A^*$, the free monoid over $\A$ (i.e.\ the set of finite words over $\A$), as well as to $\A^{\N}$, the collection of (one-sided) infinite words over $\A$. There exists a unique element 
\[
u
=
\a\b\a\a\b\a\b\a\ldots \in \A^{\N}
\]
with the property that $u = S(u)$. It is straightforward to verify that for $n \in \N$, $S^n(\a)$ is a prefix of $S^{n+1}(a)$. Thus, one obtains $u$ as the limit (in the product topology on $\mathcal{A}^\mathbb{N}$) of the sequence of finite words $\{ S^n(\a) \}_{n \in \N}$. With this setup, the Fibonacci subshift is defined to be the collection of two-sided infinite words with the same local factor structure as $u$, that is,
\[
\Omega
\eqdef
\{ \omega \in \mathcal{A}^\Z : \text{every finite subword of $\omega$ is also a subword of } u \}.
\]
Given $\ell_\a, \ell_\b \in \R_+$ and real-valued functions $f_\a \in L^2[0,\ell_\a)$, $f_\b \in L^2[0,\ell_\b)$, we consider the family of continuum Schr\"odinger operators $\{H_\omega\}_{\omega \in \Omega}$ defined by \eqref{eq:Schrod} and \eqref{eq:VomegaDef}. Since $(\Omega,T)$ is a minimal dynamical system, one can verify that there is a uniform closed set $\Sigma = \Sigma(f_\a,f_\b) \subset \R$ with the property that
\[
\sigma(H_\omega)
=
\Sigma
\text{ for every }
\omega \in \Omega.
\]

Of course, one can choose $f_\a$ and $f_\b$ in such a way that every $V_\omega$ is a periodic potential (notice that as soon as $V_{\omega_0}$ is periodic for a single $\omega_0 \in \Omega$, then every $V_\omega$ is periodic by minimality). The main result of \cite{DFG2014} is that this is the only possible obstruction to Cantor spectrum. Concretely, we assume that the system $(\Omega,f_\a,f_\b)$ is \emph{aperiodic} in the sense that the potentials $V_\omega$ described above are not periodic.

\begin{theorem}[Damanik--F.--Gorodetski \cite{DFG2014}]
Let $\Omega$ denote the Fibonacci subshift over $\A = \{\a,\b\}$. If $(\Omega,f_\a,f_\b)$ is  aperiodic, then $\Sigma$ is a Cantor set of zero Lebesgue measure.
\end{theorem}

\begin{remark}
In \cite{DFG2014}, the authors also assumed a condition on $\{f_\alpha:\alpha\in\A\}$ that they called \emph{irreducibility}. This condition is defined so that the potentials satisfy the \emph{simple finite decomposition property} (SFDP) from \cite{KLS2011}. However, since our alphabet only has two letters, SFDP follows from aperiodicity and \cite[Proposition~3.5]{KLS2011}. 
\end{remark}


Spectral properties of the family $\{H_\omega\}_{\omega \in \Omega}$ are encoded by dynamical characteristics of an associated polynomial map $T:\R^3 \to \R^3$. Concretely, every energy $E \in \R$ corresponds to a point $\gamma(E) \in \R^3$ via a (model dependent) map, $\gamma$, called the \emph{curve of initial conditions}. An energy $E$ belongs to $\Sigma$ if and only if $\gamma(E)$ has a bounded forward orbit under the action of $T$. We will describe this correspondence (and precisely define $\gamma$ and $T$) in Section~\ref{sec:bg}.	

Throughout the remainder of the paper, we assume that $\ell_\a, \ell_\b > 0$, $f_\a \in L^2[0,\ell_\a)$, and $f_\b \in L^2[0,\ell_\b)$ are given so that $(\Omega,f_\a,f_\b)$ is aperiodic. It will actually be quite helpful to introduce a third pair $(\ell_\ab,f_\ab)$, defined by
\[
\ell_\ab = \ell_\a + \ell_\b,
\quad
f_\ab(x)
=
\begin{cases}
f_\a(x) & 0 \leq x < \ell_\a \\
f_\b(x-\ell_\a) & \ell_\a \leq x < \ell_\ab.
\end{cases}
\]
Equivalently, using the notation from \eqref{eq:finiteConcatDef}, we could write $f_\ab = (\, \boxed{f_\a} \, | f_\b)$. We denote by $\widetilde\A = \set{\a,\b,\ab}$ the enlarged ``alphabet''. 

Our main results show that the local Hausdorff dimension of the spectrum tends to one in the high-energy and small-coupling regimes, regardless of the shape of $f_{\a}$ and $f_{\b}$. This answers \cite[Question~6.8]{DFG2014} in the affirmative in those two asymptotic regimes in full generality.

\begin{theorem}\label{t:largeE}
The local Hausdorff dimension of $\Sigma$ tends to one in the high-energy region. That is,
\[
\lim_{K \to \infty}\ \ \inf_{E \in \Sigma\cap[K,\infty)}\ \dim_\Hd^\loc(\Sigma;E)
=
1.
\]

\end{theorem}

\begin{theorem} \label{t:smalllambda}
With notation as above, let $\Sigma_\lambda = \Sigma(\lambda f_\a, \lambda f_\b)$ for $\lambda \in \R$. We have
\[
\lim_{\lambda \downarrow 0}\ \inf_{E \in \Sigma_\lambda \cap \R_+} \dim_\Hd^\loc(\Sigma_\lambda;E)
=
1.
\]
\end{theorem}

The rest of the paper is laid out as follows. In Section~\ref{sec:bg}, we describe some background information, including the trace-map formalism for the operator family $\{H_\omega\}_{\omega \in \Omega}$. We prove Theorem~\ref{t:largeE} in Section~\ref{sec:highE} and we prove Theorem~\ref{t:smalllambda} in Section~\ref{sec:smalllambda}.

\section{Trace Map, Fricke-Vogt Invariant, and Local Hausdorff Dimension of the Spectrum} \label{sec:bg}

The spectrum (and many spectral characteristics) of the continuum Fibonacci model can be encoded in terms of an associated polynomial diffeomorphism of $\R^3$, called the \emph{trace map}. We will make this correspondence explicit, following \cite{DFG2014}. First, we consider the differential equation
\begin{equation} \label{eq:TimeIndepSchrod}
-y''(x) + f_\alpha(x) y(x) 
= 
E y(x),
\quad
\alpha \in \widetilde\A,\;
E \in \R, \;
x \in [0,\ell_\alpha].
\end{equation} 
Denote the solution of \eqref{eq:TimeIndepSchrod} obeying $y(0) = 0$, $y'(0) = 1$ (respectively, $y(0) = 1$, $y'(0) = 0$) by $\psi_{\alpha,\Dir}(\cdot,E)$ (respectively, $\psi_{\alpha,\Neu}(\cdot,E)$).  The associated transfer matrices are then given by
\begin{align*}
M_\alpha(E) 
& = 
\begin{bmatrix} 
\psi_{\alpha,\Neu}(\ell_\alpha,E)  & \psi_{\alpha,\Dir}(\ell_\alpha,E) \\ 
\psi_{\alpha,\Neu}'(\ell_\alpha,E) & \psi_{\alpha,\Dir}'(\ell_\alpha,E) 
\end{bmatrix}, 
\quad
 \alpha \in \widetilde\A, \; E \in \R,
\end{align*}
and
\[
x_{\alpha}(E)
=
\frac{1}{2}
\tr(M_\alpha(E))
=
\frac{1}{2}\left(
\psi_{\alpha,\Neu}(\ell_\alpha,E) + \psi_{\alpha,\Dir}'(\ell_\alpha,E) \right),
\quad
\alpha \in \widetilde\A, \; E \in \R.
\]
It is straightforward to verify that
\[
M_\ab(E)
=
M_\b(E)M_\a(E).
\]
The map $\gamma(E) \eqdef (x_\ab(E), x_\a(E), x_\b(E))$ is known as the \emph{curve of initial conditions}. Then, the \emph{trace map} is defined by
$$
T(x,y,z) 
\eqdef
(2 xy-z,x,y),
\quad
x,y,z \in \R.
$$
This map is known to have a first integral given by the so-called \emph{Fricke--Vogt invariant}, defined by
$$
I(x,y,z) \eqdef x^2 + y^2 + z^2 - 2 xyz -1,
\quad
x,y,z \in \R.
$$
More precisely, the trace map $T$ preserves $I$ (in the sense that $I\circ T = I$), and hence $T$ preserves the level surfaces of $I$:
$$
S_V
\eqdef
\{ (x,y,z) \in \R^3 : I(x,y,z) = V \}.
$$
Consequently, every point of the form $T^n(x_\ab(E), x_\a(E), x_\b(E))$ with $n \in \Z_+$ lies on the surface $S_{I(\gamma(E))}$. For the sake of convenience, we put
$$
I(E) 
\eqdef
I(\gamma(E))
=
I(x_\ab(E), x_\a(E), x_\b(E)),
$$
with a minor abuse of notation.

The surfaces $S_V$ experience a transition at $V = 0$. When $V<0$, $S_V$ has one compact connected component which is diffeomorphic to the 2-sphere $\S^2$, and four unbounded connected components, each of which is diffeomorphic to the open unit disk. When $V = 0$, each of the four unbounded components meet the compact component, forming four conical singularities. As soon as $V >0$, the singularities resolve; then, $S_V$ is smooth, connected, and diffeomorphic to $\S^2$ with four points removed.

The trace map is important in the study of operators of this type as its dynamical spectrum encodes the operator-theoretic spectrum of $H$. Concretely, the dynamical spectrum is defined by
$$
B 
\eqdef
\set{ E \in \R : \{ T^n(\gamma(E)) : n \in \Z_+ \} \text{ is bounded} }.
$$

\begin{prop}[Damanik--F.--Gorodetski \cite{DFG2014}]\label{p.sigmab}
We have $\Sigma = B$.
\end{prop}

There are several substantial differences between the continuum setting and the discrete setting that we should point out. First, in the discrete case, the Fricke--Vogt invariant is constant (viewed as a function of $E$). However, the invariant may enjoy nontrivial dependence on $E$ in the continuum setting, which is demonstrated by examples in \cite{DFG2014}. This dependence is related to new phenomena that emerge in the continuum setting and make its study worthwhile.

Second, the Fricke--Vogt invariant is always non-negative in the discrete setting, but one cannot \textit{a priori} preclude negativity of $I$ in the continuum setting. However, it is proved in \cite{DFG2014} that any energies for which $I(E) < 0$ must lie in the resolvent set of the corresponding continuum Fibonacci Hamiltonian.

\begin{prop}[Damanik--F.--Gorodetski \cite{DFG2014}]\label{p.nonneginv}
For every $E \in \Sigma$, one has $I(E) \ge 0$
\end{prop}

In order to study the fractal dimension of the spectrum, we will use the following theorem from \cite{DFG2014}, which connects local fractal characteristics near an energy $E$ in the spectrum with the value of the invariant at $E$.

\begin{theorem}[Damanik--F.--Gorodetski \cite{DFG2014}] \label{t.fiblocaldim}
There exists a continuous map $D : [0,\infty) \to (0,1]$ that is real-analytic on $\R_+$ with the following properties:
\begin{itemize}

\item[{\rm (i)}] $\dim_\Hd^\loc(\Sigma;E) = D(I(E))$ for each $E \in \Sigma$.

\item[{\rm (ii)}] We have $D(0) = 1$ and $1 - D(I) \sim \sqrt{I}$ as $I \downarrow 0$.

\item[{\rm (iii)}] As $I \to \infty$, we have
$$
\lim_{I \to \infty} D(I) \cdot \log I = 2 \log (1 + \sqrt{2})
$$

\end{itemize}
\end{theorem}

Thus, to study the local fractal dimensions of the spectrum, it suffices to understand the the invariant $I$.

\section{The High-Energy Region} \label{sec:highE}

\begin{proof}[Proof of Theorem~\ref{t:largeE}]
In view of Theorem~\ref{t.fiblocaldim}, it suffices to show that 
\begin{equation} \label{eq:invDecayInE}
\lim_{E \to \infty} I(E) 
= 
0.
\end{equation}
In fact, we will show that $I(E)$ goes to zero as $\mathcal O(E^{-1/2})$, where the implicit constant only depends on $\ell_\a$, $\ell_\b$, $\|f_\a\|$, and $\|f_\b\|$. Here and in what follows, we use $\| \cdot \|$ to denote an $L^2$ norm on an appropriate interval, i.e.,
\[
\|f_\alpha\|
=
\left(\int_0^{\ell_\alpha} |f_\alpha(x)|^2 \, \dd x \right)^{1/2},
\quad
\alpha \in \widetilde\A.
\]
Concretely, let
\[
Q
\eqdef
\max\left(\|f_\ab\|,1\right),
\quad
C 
\eqdef
Q\exp\left(Q \sqrt{\ell_\ab} \right).
\]
For each energy $E \ge 0$, we denote
\[
\kappa = \kappa(E) = \sqrt{E},
\]
and we introduce functions
\[
c_\kappa(x)
\eqdef
\cos(\kappa x),
\quad
s_\kappa(x)
\eqdef
\frac{\sin(\kappa x)}{\kappa},
\quad
x \in \R.
\]
Then, by \cite[Theorem~1.3]{PoTru1987}, one has the following estimates for $\alpha \in \widetilde\A$, $E > 0$, $x \in [0,\ell_\alpha]$:

\begin{align*}
\left| \psi_{\alpha,\Neu}(x,E) - c_\kappa(x) \right|
& \leq 
\frac{1}{\sqrt{E}} \exp\left( \left|\mathrm{Im}\sqrt{E} \right| x + \|f_\alpha\| \sqrt{x} \right) \\
& \leq 
\frac{1}{\sqrt{E}} \exp\left( \|f_\alpha\| \sqrt{\ell_\alpha} \right) \\
& \leq
C\kappa^{-1}
\end{align*}
Similarly, we get
\[
\left| \psi_{\alpha,\Dir}'(x,E) - c_\kappa(x) \right|
\leq 
C\kappa^{-1}
\]
for every $\alpha\in\widetilde{\mathcal A}$, $E > 0$, and $x \in [0,\ell_\alpha]$.
In particular, we get
\[
x_\alpha(E) = c_\kappa(\ell_\alpha) + \mathcal O(\kappa^{-1})
\]
for $\alpha \in \widetilde\A$ and $E > 0$ large. Consequently,
the asymptotics above yield
\begin{align*}
I(E) 
& =
I(x_\ab(E), x_\a(E), x_\b(E)) \\
& =
I(c_\kappa(\ell_\ab), \, c_\kappa(\ell_\a), \, c_\kappa(\ell_\b)) 
+ \mathcal O\left(\kappa^{-1}\right).
\end{align*}
A straightforward calculation reveals that
\[
I(c_\kappa(\ell_\ab), \, c_\kappa(\ell_\a), \, c_\kappa(\ell_\b)) 
=
0.
\]
Thus, $I(E) = \mathcal O(E^{-1/2})$, as desired.
\end{proof}

\begin{remark}
Since we obtain $I(E) = \mathcal O(E^{-1/2})$ as $E \to \infty$, we can actually use Theorem~\ref{t.fiblocaldim} to obtain a quantitative lower bound on the rate at which the Hausdorff dimension of the spectrum tends to one. Namely, there exists a constant $c = c(f_\a,f_\b) > 0$ with the property that
\[ 
\dim_\Hd^\loc(\Sigma;E) 
\geq 
\max\left(1 - cE^{-1/4},0\right)
\]
for all $E \in \Sigma \cap \R_+$.
\end{remark}

\section{The Small Coupling Regime} \label{sec:smalllambda}

For the regime of small coupling, we fix $f_\a$ and $f_\b$ as before so that $(\Omega,f_\a,f_\b)$ is aperiodic, fix a Fibonacci-type potential $V_\omega$, and ask about the behavior of
\[
H_{\omega,\lambda}
=
-\frac{\dd^2}{\dd x^2} + \lambda V_\omega
\]
as $\lambda \to 0$. We now view the relevant spectral data as functions of $\lambda$ as well, and we write $\Sigma_\lambda$ for the common spectrum of $H_{\omega,\lambda}$. Similarly, we view the Dirichlet and Neumann solutions as functions of $\lambda$, so we write $\psi_{\alpha,\mathrm B}(\cdot,E,\lambda)$ for the solution of
\[
-y'' + \lambda f_\alpha y = Ey
\]
obeying the appropriate boundary condition for $\mathrm B \in \set{\Dir,\Neu}$. Then, the transfer matrices and invariant are also functions of $\lambda$, and we denote them by $M_\alpha(E,\lambda)$ and $I(E,\lambda)$ to reflect this dependence.

\begin{proof}[Proof of Theorem~\ref{t:smalllambda}]

By Theorem~\ref{t.fiblocaldim}, it suffices to show that $I(E,\lambda)$ goes to zero uniformly in $E > 0$ as $\lambda \to 0$. To that end, let $\e > 0$ be given. By the estimates in the proof of Theorem~\ref{t:largeE}, we may choose a compact set $K \subset \R$ such that
\[
|I(E,\lambda)|
< \e
\]
whenever $|\lambda| \leq 1$ and $E \in \R_+ \setminus K$. By \cite[Theorem~5]{PoTru1987}, we have that
\[
\psi_{\alpha,\Neu}(\ell_\alpha,E,\lambda)
\to
c_\kappa(\ell_\alpha)
\]
as $\lambda \to 0$, uniformly for $\alpha \in \widetilde \A$, $E \in K$. Thus, we may choose $\lambda_0 > 0$ so that
\begin{align}
\label{eq:smallcoup:solest}
\left| \psi_{\alpha,\Neu}(\ell_\alpha,E,\lambda) - c_\kappa(\ell_\alpha) \right|
& < 
\e
\end{align}
for all $\alpha \in \widetilde \A$ and $E \in K$  whenever $|\lambda| < \lambda_0$. 

Differentiating the integral equation from \cite[Theorem~1]{PoTru1987} , we get that
\[
\psi_{\alpha,\Dir}'(\ell_\alpha,E,\lambda) 
= 
c_\kappa(\ell_\alpha) + \int_{0}^{\ell_\alpha} \cos(\kappa(\ell_\alpha-t)) \cdot \lambda f_{\alpha}(t) \cdot \psi_{\alpha,\Dir}(t,E,\lambda)\ dt.
\]
Again by \cite[Theorem~5]{PoTru1987}, $\psi_{\alpha,\Dir}(t,E,\lambda) \to s_\kappa(t)$ as $\lambda \to 0$, uniformly for $E \in K$ and $t \in [0,\ell_\alpha]$. Consequently (possibly after shrinking $\lambda_0$) we obtain
\begin{align}
\label{eq:smallcoup:solestD}
\left| \psi_{\alpha,\Dir}'(\ell_\alpha,E,\lambda)- c_\kappa(\ell_\alpha) \right|
& < 
\e
\end{align}
for all $\alpha \in \widetilde \A$ and $E \in K$ whenever $|\lambda| < \lambda_0$. Therefore, by (\ref{eq:smallcoup:solest}) and (\ref{eq:smallcoup:solestD}),
\[
I(E,\lambda) = I(E,0) + \mathcal O(\varepsilon)
\]
for $E \in K$ and $|\lambda| < \lambda_0$, with a uniform implicit constant. Since $I(E,0) = 0$, we have $I(E,\lambda) = \mathcal O(\e)$ for all $|\lambda| < \lambda_0$, as desired.
\end{proof}

\section*{Acknowledgements}
M.\ M.\ would like to thank Virginia Tech for hosting her for the Fall 2016 semester, when much of this work was conducted. The authors gratefully acknowledge Mark Embree for helpful conversations.


\begin{thebibliography}{00}

\bibitem{BBG2015} M.\ Baake, M.\ Birkner, U.\ Grimm, Non-periodic systems with continuous diffraction measures, \textit{Mathematics of aperiodic order}, 1--32, \textit{Prog.\ Math.\ Phys.}\ \textbf{309}, Birkh\"auser/Springer, Basel, 2015. 

\bibitem{BJK1992}  M.\ Baake, D.\ Joseph, P.\ Kramer, Periodic clustering in the spectrum of quasiperiodic Kronig-Penney models, \textit{Phys.\ Lett.\ A} \textbf{168} (1992), 199--208.

\bibitem{Cas1986CMP} M.\ Casdagli, Symbolic dynamics for the renormalization map of a quasi-periodic Schr\"odinger equation, \textit{Commun.\ Math.\ Phys.}\ \textbf{107} (1986), 295--318.

\bibitem{Damanik2007:subshifts} D.\ Damanik, Strictly ergodic subshifts and associated operators, in \textit{Spectral Theory and Mathematical Physics: a Festschrift in Honor of Barry Simon's 60th Birthday}, 505--538, Proc.\ Sympos.\ Pure Math.\ \textbf{76}, Part 2, Amer. Math. Soc., Providence, RI, 2007.

\bibitem{DEG2015} D.\ Damanik, M.\ Embree, A.\ Gorodetski, Spectral properties of Schr\"odinger operators arising in the study of quasicrystals, \textit{Mathematics of aperiodic order}, 307--370, \textit{Prog.\ Math.\ Phys.}\ \textbf{309}, Birkh\"auser/Springer, Basel, 2015. 

\bibitem{DFG2014} D.\ Damanik, J.\ Fillman, A.\ Gorodetski, Continuum Schr\"odinger operators associated with aperiodic subshifts, \textit{Ann.\ Henri Poincar\'e} , \textbf{15} (2014), 1123--1144.

\bibitem{DFO2016} D.\ Damanik, J.\ Fillman, D.\ C.\ Ong, Spreading Estimates for Quantum Walks on the Integer Lattice via Power-Law Bounds on Transfer Matrices, \textit{J.\ Math.\ Pures Appl.}\ \textbf{105} (2016), 293--341.

\bibitem{DFV2014} D.\ Damanik, J.\ Fillman, R.\ Vance, Dynamics of unitary operators, \textit{J.\ Fractal Geom.}\ \textbf{1} (2014), 391--425.

\bibitem{DGY} D.\ Damanik, A.\ Gorodetski, W.\ Yessen, The Fibonacci Hamiltonian, \textit{Invent.\ Math.}\ \textbf{206} (2016), 629--692.

\bibitem{DamLenz2007} D.\ Damanik, D.\ Lenz, Uniform Szeg\H{o} cocycles over strictly ergodic subshifts, \textit{J.\ Approx.\ Theory} \textbf{144} (2007), 133--138.

\bibitem{DMY2013JAT} D.\ Damanik, P.\ Munger, W.\ Yessen, Orthogonal polynomials on the unit circle with Fibonacci Verblunsky coefficients, I.~The essential support of the measure,  \textit{J.\ Approx.\ Th.}\ \textbf{173} (2013), 56--88.

\bibitem{DMY2013JSP} D.\ Damanik, P.\ Munger, W.\ Yessen, Orthogonal polynomials on the unit circle with Fibonacci Verblunsky coefficients, II. Applications, \textit{J.\ Stat.\ Phys.}\ \textbf{153} (2013), 339--362.

\bibitem{F2017PAMS} J.\ Fillman, Purely singular continuous spectrum for Sturmian CMV matrices via strengthened Gordon Lemmas, \textit{Proc.\ Amer.\ Math.\ Soc.}\ \textbf{145} (2017), 225--239.

\bibitem{FTY2016} J.\ Fillman, Y.\ Takahashi, W.\ Yessen, Mixed spectral regimes for square Fibonacci Hamiltonians, \textit{J.\ Fract.\ Geom.}, \textbf{3} (2016), 377--405.

\bibitem{Ghosh1991} P.\ Ghosh, The Kronig-Penney model on a generalized Fibonacci lattice, \textit{Phys.\ Lett.\ A} \textbf{161} (1991), 153--157.

\bibitem{Holzer1988} M.\ Holzer, Three classes of one-dimensional, two-tile Penrose tilings and the Fibonacci Kronig-Penney model as a generic case, \textit{Phys.\ Rev.\ B}  \textbf{38} (1988), 1709--1720.

\bibitem{KN2002} M.\ Kaminaga, F.\ Nakano, Spectral properties of quasiperiodic Kronig-Penney model, \textit{Tsukuba J.\ Math.}\ \textbf{26} (2002), 205--228.

\bibitem{KLS2011} S.\ Klassert, D.\ Lenz, P.\ Stollmann, Delone measures of finite local complexity and applications to spectral theory of one-dimensional continuum models of quasicrystals, \textit{Discrete Contin.\ Dyn.\ Syst.}\ \textbf{29} (2011), 1553--1571.

\bibitem{KKT1983} M.\ Kohmoto, L.\ P.\ Kadanoff, C.\ Tang, Localization problem in one dimension: mapping and escape, \textit{Phys.\ Rev.\ Lett.}\ \textbf{50} (1983), 1870--1872.

\bibitem{KS1986} J.\ Kollar, A. S\"ut\H{o}, The Kronig-Penney model on a Fibonacci lattice, \textit{Phys.\ Lett.\ A} \textbf{117} (1986), 203--209.

\bibitem{LSS2014} D.\ Lenz, C.\ Seifert, P.\ Stollman, Zero measure Cantor spectra for continuum one-dimensional quasicrystals, \textit{J.\ Diff.\ Eq.}\ \textbf{256} (2014), 1905--1926.

\bibitem{MO2014} P.\ Munger, D.\ Ong, The H\"older continuity of spectral measures of an extended CMV matrix, \textit{J.\ Math.\ Phys.}\ \textbf{55} (2014), no.\ 9, 093507, 10 pp. 

\bibitem{OPRSS1983} S.\ Ostlund, R.\ Pandit, D.\ Rand, H.\ J.\ Schnellnhuber, E.\ D.\ Siggia, One-dimensional Schr\"odinger equation with an almost-periodic potential, \textit{Phys.\ Rev.\ Lett.}\ \textbf{50} (1983), 1873--1876.

\bibitem{PoTru1987} J.\ P\"oschel, E.\ Trubowitz, \textit{Inverse Spectral Theory}, Academic Press.

\bibitem{Suto87} A.\ S\"ut\H{o}, The spectrum of a quasiperiodic Schr\"odinger operator, \textit{Commun.\ Math.\ Phys.}\ \textbf{111} (1987), 409--415.

\bibitem{Suto89} A.\ S\"ut\H{o}, Singular continuous spectrum on a Cantor set of zero Lebesgue measure for the Fibonacci Hamiltonian, \textit{J.\ Stat.\ Phys.} {\bf 56} (1989), 525--531.

\bibitem{SBGC84}D.\ Shechtman, I.\ Blech, D.\ Gratias, and J.W.\ Cahn.
Metallic phase with long-range orientational order and no  translational symmetry,
\textit{Phys. Rev. Lett.}, {\bf 53.20} (1984), 1951--1953.

\bibitem{TK1989} U.\ Thomas, P.\ Kramer, The Fibonacci quasicrystal reconsidered: variety of energy spectra for continuous Schr\"odinger equations with simple potentials, \textit{Internat.\ J.\ Modern Phys.\ B} \textbf{3} (1989), 1205--1235.

\bibitem{WSS1988} D.\ W\"urtz,  M.\ Soerensen, T.\ Schneider, Quasiperiodic Kronig-Penney model on a Fibonacci superlattice, \textit{Helv.\ Phys.\ Acta} \textbf{61} (1988), 345--362.

\bibitem{Yessen2013} W.\ Yessen, Spectral analysis of tridiagonal Fibonacci Hamiltonians, \textit{J.\ Spectr.\ Theory}, \textbf{3} (2013), 101--128.

\end{thebibliography}
\end{document}